\documentclass{svjour3}                     
\smartqed  
\usepackage{graphicx}

\usepackage{color}
\usepackage{graphicx,color}
\usepackage{graphicx} 
\usepackage{amsmath, amssymb, graphics}

\def \r{\mathbb R}

\begin{document}

\title{Translating solitons of translation and homothetical types\thanks{Rafael L\'opez has been partially supported by the grant no. MTM2017-89677-P, MINECO/ AEI/FEDER, UE.}
}

\author{Muhittin Evren Aydin         \and
        Rafael L\'opez 
}

\institute{Muhittin Evren Aydin \at
              Department of Mathematics\\ Faculty of Science, Firat University\\ Elazig, 23200, Turkey\\
              \email{meaydin@firat.edu.tr}           
           \and
           Rafael L\'opez \at
              Departamento de Geometr\'{\i}a y Topolog\'{\i}a\\
 Universidad de Granada\\
 18071 Granada, Spain
 \email{rcamino@ugr.es}
}

\date{Received: date / Accepted: date}

\maketitle

\begin{abstract}
We prove that if a translating soliton  can be expressed as the sum of two curves and one of these curves is planar, then the other curve is also planar and consequently  the surface must be  a plane or a grim reaper. We also investigate translating solitons that can be locally written as the product of two functions of one variable. We extend the results  in Lorentz-Minkowski space.
\end{abstract}

\section{Introduction}

Let  $\vec{v}\in\r^3$ be non-zero vector. A {\it translating soliton} in Euclidean $3$-dimensional space $\r^3$ with respect to $\vec{v}$, called the {\it velocity} of the flow, is a surface $M$ whose mean curvature $H$ satisfies
\begin{equation}\label{eq1}
H(p)=\langle N(p),\vec{v}\rangle, 
\end{equation}
for all $p\in M$, where $N$ is the unit normal vector field on $M$.   Translating solitons appear in the  theory of the mean curvature flow  of Huisken and Ilmanen as the solutions of the flow when $M$ evolves purely by translations along the direction $\vec{v}$ (\cite{hsi1,il}). In particular,    $M+t\vec{v}$, $t\in \r$, satisfies that fixed $t$,   the normal component of the velocity vector $\vec{v}$   at each point is equal to the mean curvature at that point. In non-parametric way $z=u(x,y)$, Equation \eqref{eq1} is
\begin{equation}\label{eq2}
(1+u_y^2)u_{xx}-2u_xu_y u_{xy}+(1+u_x^2)u_{yy}=2(1+u_x^2+u_y^2)(-v_1 u_x-v_2 u_y+v_3),
\end{equation}
where the subindices indicate the corresponding partial differentiation and $\vec{v}=(v_1,v_2,v_3)$. This equation is a quasilinear  elliptic PDE, so the solvability is not assured. Some  results of the solvability  of the Dirichlet problem can be proved assuming convexity in the initial data (\cite{Lopez5}). A way to reduce the complexity of \eqref{eq2} is assuming some type of symmetry on the surface which makes that \eqref{eq2} converts into an ordinary differential equation, where classical theory ensures local existence of solutions. Following this strategy, we can assume that the surface is invariant under a uniparametric group of translations (cylindrical surfaces) or rotations (surfaces of revolution). Both families of surfaces are classified and   play  a remarkable role in the theory of translating solitons. We now describe both examples. Let $(x,y,z)$ be  the canonical coordinates of $\r^3$.
\begin{enumerate}
\item Cylindrical surfaces. The translating solitons are planes parallel to the velocity vector $\vec{v}$ if the rulings are parallel to $\vec{v}$ or grim reapers otherwise. See a detailed discussion in \cite{Lopez5}. For grim reapers, we can assume without  loss of generality that $\vec{v}=(0,0,1)$. Let $w$ be the direction of the rulings. After a rotation about the $z$-axis, let $w=\cos\theta e_1+\sin\theta e_3$, where $\{e_1,e_2,e_3\}$ is the canonical basis of $\r^3$, $e_3=\vec{v}$ and $\cos\theta\not=0$. The generating curve is included in the plane spanned by $\{e_2,e\}$, where  $e=-\sin\theta e_1+\cos\theta e_3$. If we write this curve as $\beta(s)=se_2+u(s)e$, then $u$ satisfies $u''=\cos\theta(1+u'^2)$. This equation can be completely integrated. For example, if $\theta=0$, $u(s)=-\log(\cos(s+a))+b$, $a,b\in\r$.  We point out that if a translating soliton is a ruled surface, then it must be cylindrical, hence a plane or a grim reaper  (\cite{Hieu}).
\item  Surfaces of revolution. The rotation axis is not arbitrary and must be parallel to the velocity vector $\vec{v}$. There are two types of rotational translating solitons depending on whether or not the surface meets  the rotation axis (\cite{aw,css}). In the first case, the surface is known in the literature as the bowl soliton and in the second one, the surfaces have winglike shape.  
\end{enumerate}

Another way to address Equation \eqref{eq2} is by the technique of separation of variables. We have two possibilities, $u(x,y)=f(x)+g(y)$ and $u(x,y)=f(x)g(y)$, where $f$ and $g$ are smooth functions of one variable. In both cases, Equation \eqref{eq2} is an ODE where the unknowns are the functions $f=f(x)$ and $g=g(y)$. If $z=f(x)+g(y)$, the translating soliton equation \eqref{eq2} is now
 \begin{equation}\label{eq3}
(1+g'^2)f''+ (1+f'^2)g''=(1+f'^2+g'^2)(-v_1 f'-v_2 g'+v_3),
\end{equation}
where $'$ indicates the derivative with respect to the corresponding variable. In \cite{Lopez3}, the second author proved that if $\vec{v}=(0,0,1)$, grim reapers are the only solutions of \eqref{eq3}. Let us observe that the planes parallel to $\vec{v}$ are not graphs on the $xy$-plane.  A surface that is the graph of $z=f(x)+g(y)$ can be expressed as the    sum of two planar curves $\alpha(x)+\beta(y)$, where $\alpha(x)=(x,0,f(x))$ and $\beta(y)=(0,y,g(y))$. Let us observe that both curves are  contained in orthogonal planes. More generally, a surface is said to be a {\it translation surface} if it is the sum of two curves   called  {\it generating curves}. The name of translation surface is due to the surface can be viewed from the kinematic viewpoint as the translation of the curve $\alpha$ (of $\beta$) by means of translations through $\beta$ (or $\alpha$, respectively). Thus the result in  \cite{Lopez3} is only a partial answer to the following

\begin{quote}{\bf Problem 1.} Classify all translating solitons that are translation surfaces.
\end{quote}

This problem has its analogy in the classical theory of minimal surfaces of $\r^3$. Scherk proved that besides the planes, the only minimal surface that can be expressed as $z=f(x)+g(y)$ is 
$$u(x,y)=\frac{1}{c}\log\left|\frac{\cos(cy)}{\cos(c x)}\right|,$$
where $c\not=0$ (\cite{sc}). More recently, Dillen {\it et al.}  proved that if  one of the generating curves of a minimal surface of translation type is planar, then the other generating curve is also planar (\cite{Dillen}) and the surface belongs to a family of minimal surfaces discovered by Scherk (\cite{ni}). Surprisingly, very recently the second author together Hasanis and Perdomo discovered many minimal surfaces of translation type where both generating curves are not planar (\cite{lh,lp}).

In this paper we follow the same approach  for translating solitons. However,  the presence of the vector $\vec{v}$ in Equation \eqref{eq1} makes  a great difference because $\vec{v}$ is an arbitrary vector in relation with the spatial coordinates $(x,y,z)$ of  $\r^3$.  We give a partial answer to Problem 1 assuming that one of the generating curves is planar and proving that the surface is a plane or a grim reaper (Theorem \ref{t2}).   As a previous step, we prove this result in case that the two generating curves are planar curves but not necessarily contained in orthogonal planes (Theorem \ref{t1}). Both results are analogous to the minimal surfaces obtained in \cite{Dillen}. The goal of all both theorems is that  we do not presuppose any relation between the velocity vector $\vec{v}$ and the surface. More precisely, the notion of translation surface is affine because we use the sum of vectors of $\r^3$.  However the velocity vector $\vec{v}$ in the translating soliton equation \eqref{eq1} is assumed in all its generality without any relation with the coordinates of $\r^3$. This should be pointed out because one may be tempted to fix $\vec{v}$ since Equation \eqref{eq1} is invariant after a rigid motion.   However, such rigid motion also changes the spatial coordinates of $\r^3$. This seems to be subtle, but if one assumes that the surface is $z=f(x)+g(y)$, then the vector $\vec{v}$ must be arbitrary. All this complicates the demonstrations, which are not straightforward.
 
  The second case of separation of variables that we investigate for the translating soliton equation is $z=f(x)g(y)$. Then \eqref{eq2} is 
   \begin{equation}\label{eqh1}
 (1+f^2g'^2)gf''-2fgf'^2g'^2+(1+g^2f'^2)fg''-2(1+f'^2g^2+f^2g'^2)(-v_1f'g-v_2fg'+v_3)=0.
 \end{equation}
 Let us observe the symmetry of \eqref{eqh1} in terms of $f$ or $g$, hence any discussion on one of both functions also holds for the other one.     As far as the authors  know,  the first approach to this kind of surfaces in relation to the study of the curvature of surfaces appeared in \cite{wo1,wo2}, where the authors coined this type of surfaces as {\it homothetical surfaces} (see also \cite{GV,LM}).   We have the analogous question. 

\begin{quote}{\bf Problem 2.} Classify all translating solitons of homothetical type.
\end{quote}

It was proved in \cite{Lopez3}, and in the particular case $\vec{v}=(0,0,1)$, that the only homothetical translating solitons are grim reapers. Grim reapers appear when one of the functions $f$ or $g$ are constant. Indeed, if  say $f(x)=a$, $a\in\r$, then the parametrization of the surface is $X(x,y)=(x,y,a g(y))$ deducing that the surface is cylindrical and the rulings are parallel to the vector $e_1$ of the canonical basis. In contrast to Equation \eqref{eq3}, now Equation \eqref{eqh1} is more difficult to work. The result that we prove is assuming that $\vec{v}$ is one of the canonical directions of $\r^3$ and proving that the surface is a plane or a grim reaper (Theorem \ref{t3}). Again we can make the same observation as before and although this seems elementary analysis and would yield no non-trivial solutions besides cylindrical surfaces, one can expect the existence of new examples. For instance, in the family of minimal surfaces, the plane and the helicoid (which is not cylindrical but ruled) are the only homothetical surfaces (\cite{wo1}). But if one replaces $z=f(x)g(y)$ by $h(z)=f(x)g(y)$, then there are many minimal surfaces (\cite{ni}). 

Finally in Section \ref{s4} we extend all the above results for translating solitons in Lorentz-Minkowski space $\r_1^3$. Since the underlying affine space for $\r_1^3$ coincides with  the Euclidean space, the concepts of translation surfaces and homothetical surfaces are equally valid in the Lorentzian setting. The results are analogous to that of Euclidean space.

\section{Translating solitons of translation type}

Consider a translation surface where the generating curves are planar curves.   If the planes containing the generating curves are orthogonal and  $\vec{v}$ is parallel to both planes,   the second author proved that the only translating solitons are grim reapers whose rulings are parallel to one of the above planes (\cite{Lopez5}).  We now investigate Problem 1 in case that $\vec{v}$ is arbitrary and the generating curves are planar but not necessarily containing in orthogonal planes.

\begin{theorem}\label{t1} Planes and grim reapers are the only translating solitons that are the sum of  two planar curves.  
\end{theorem}

\begin{proof} If the planes containing the curves are parallel then the sum of the two curves is (part of) a plane.  Suppose now that both planes are not parallel. After renaming coordinates, we will assume that the $z$-axis is the common straight line of the two planes, one of the generating curves is included in the plane of equation $x=0$ and the other in the plane $cx+y=0$, $c\in\r$. The cosine
of angle between the two planes is $c/\sqrt{1+c^{2}}$ and if $c=0$
then both planes become perpendicular. The first curve parametrizes as $\beta(y)=(0,y,g(y))$ and the second one by $\alpha(x)=(x,-cx,f(x))$, where $f$ and $g$ are two smooth functions defined in  intervals $I$ and $J$, respectively. Thus a parametrization of the surface is 
\[X(x,y)=\alpha(x)+\beta(y)=(x,y-cx,f(x)+g(y)).\]
Notice that if we name $\tilde{y}=y-cx$, then the surface is $z=f(x)+g(\tilde{y}+cx)$. These surfaces are known in the literature as  affine translation surfaces (\cite{Liu2}).

In case that $f$ or $g$ is a linear function, then the surface is cylindrical and the surface must be a plane or a grim reaper, proving the result.  Now we discard this case. Then there are $x_0\in I$ and $y_0\in J$ such that $f''(x_0)\not=0$ and $g''(y_0)\not=0$. Then $f''\not=0$ and $g''\not=0$ in some subintervals around $x_0$ and $y_0$ respectively, which we can assume to be $I$ and $J$. In both intervals, there are points where $f'\not=0$ and $g'\not=0$, otherwise $f$ or $g$ would be constant functions. Abusing of notation, suppose $f'f''(x_0)\not=0$ and $g'g''(y_0)\not=0$ and analogously, $f'f''\not=0$ in $I$ and $g'g''\not=0$ in $J$. If $\vec{v}=(v_1,v_2,v_3)$, Equation \eqref{eq2} writes as 
$$
\left( 1+g^{\prime
2}\right) f^{\prime \prime }+\left( 1+c^{2}+f^{\prime 2}\right) g^{\prime \prime }=  
2\left( -v_1 \left( f^{\prime }+cg^{\prime }\right) -v_2 g^{\prime
}+v_3 \right) \left( 1+g^{\prime }{}^{2}+\left( f^{\prime }+cg^{\prime
}\right) ^{2}\right).$$
Divide by $(1+g'^2)(1+c^2+f'^2)$,
$$
\frac{ f^{\prime \prime }}{1+c^2+f'^2}+\frac{ g^{\prime \prime }}{1+g'^2}=  
2\left( -v_1 \left( f^{\prime }+cg^{\prime }\right) -v_2 g^{\prime
}+v_3 \right)\frac{ 1+g^{\prime }{}^{2}+\left( f^{\prime }+cg^{\prime
}\right) ^{2}}{(1+g'^2)(1+c^2+f'^2)}.$$

Because the left hand side is the sum of a function on the variable $x$ and a function on the variable $y$, when we differentiate with respect to $x$ and next with respect to $y$, these terms are zero.  The corresponding differentiations on the right hand side give the expression
$$\left(\sum_{n=0}^4 P_n(y){f'}^n\right)\frac{f''g''}{(1+g'^2)^2(1+c^2+f'^2)^2}=0,$$
where $P_n$ are functions on the variable $y$.  Thus $\sum_{n=0}^4 P_n(y)f'(x)^n=0$ in $I\times J$. Since this is a polynomial of the function $f'=f'(x)$, all coefficients $P_n$ must vanish in $J$. The computation of $P_4$ yields
 $P_4=-v_1 g'$, deducing $v_1=0$ because $g'\not=0$. Taking into account that $v_1=0$, the computation of $P_2$ gives  
 $$P_2=c (-v_3{g'}^2-2v_2g'+v_3).$$
 We discuss two cases:
 \begin{enumerate}
 \item Case $c=0$. Then all $P_n$ are trivially $0$ except $P_1$, which is $P_1=-g'(v_2g'^3+3v_2g'-2v_3)$. From $P_1=0$ and because $g'\not=0$, we have $v_2g'^3+3v_2g'-2v_3=0$. Since $g''\not=0$, the functions $\{1,g',g'^3\}$ are linearly independent, concluding  $v_2=v_3=0$, so  $\vec{v}=0$ obtaining a contradiction. 
 \item Case $c\not=0$. Then $-v_3g'^2-2v_2g'+v_3=0$.  Thus $v_2=v_3=0$ again, which is contradictory. 
 
 \end{enumerate}
 \end{proof}
 
 We point out that in \cite{yoon} the authors obtained a partial result of Theorem \ref{t1} in case that $\vec{v}$ is one vector of the canonical basis.

Our next progress in Problem 1 is considering that one of the generating curves is non-planar.
\begin{theorem}\label{t2}
 Planes and grim reapers are the only translating solitons that are the sum of two curves and where one of the generating curves is planar. 
\end{theorem}

\begin{proof} Suppose that the surface is parametrized by $X(s,t)=\alpha(s)+\beta(t)$, where $\beta$ is a planar curve. Without loss of generality, we assume that $\beta$  is contained in the plane $\Pi$ of equation $x=0$ and that $\beta$ parametrizes as $\beta(y)=(0,y,g(y))$, where $g$ is a smooth function defined in an interval $J$.   The proof of theorem is by contradiction so by Theorem \ref{t1}, we suppose that  the curve $\alpha$ is not planar.  Since $\alpha$ is a space curve, then $\alpha$ is a graph on one of the coordinates axes. We can assume that this axis is the $x$-axis because otherwise, the curve $\alpha$ would be contained in a plane parallel to $\Pi$ and the sum of $\alpha$ and $\beta$ would be (part of) a plane. Definitively, $\alpha$ can be expressed as  $\alpha(x)=(x,f(x),h(x))$, where $f$ and $h$ are two smooth functions defined in an interval $I\subset\r$. If we parametrize the surface by  $X(x,y)=(x,y+f(x), h(x)+g(y))$, the unit normal vector field is 
\begin{equation*}
N=\frac{1}{\sqrt{1+g^{\prime 2}+\left( f^{\prime }g^{\prime }-h^{\prime
}\right) ^{2}}}\left( f^{\prime }g^{\prime }-h^{\prime },-g^{\prime
},1\right) 
\end{equation*}
and the mean curvature $H$ is 
\begin{equation*}
H=\frac{\left( h^{\prime \prime }-f^{\prime \prime }g^{\prime }\right)
\left( 1+g^{\prime 2}\right) +\left( 1+f^{\prime 2}+h^{\prime 2}\right)
g^{\prime \prime }}{2\left( 1+g^{\prime 2}+\left( f^{\prime }g^{\prime
}-h^{\prime }\right) ^{2}\right) ^{3/2}}.
\end{equation*}%
Let $\vec{v}=(v_1,v_2,v_3)$. The translating soliton equation \eqref{eq2} is 
\begin{equation}\label{eq22}
\left( h^{\prime \prime }-f^{\prime \prime }g^{\prime }\right) \left(
1+g^{\prime 2}\right) +\left( 1+f^{\prime 2}+h^{\prime 2}\right) g^{\prime
\prime }  
=2\left(  v_1(f'g'-h')-v_{2}  g^{\prime } +v_{3}\right) \left( 1+g^{\prime 2}+\left( f^{\prime }g^{\prime }-h^{\prime
}\right) ^{2}\right). \end{equation}
If $f$ or $h$ are linear functions, then the generating curve $\alpha$ is planar, which is not possible. Therefore, with a similar argument as in the beginning of the proof of Theorem \ref{t1}, we can assume that in some subintervals  of $I$ and $J$,   we have $f'f''h'h''\not =0$  and $g'g''\not=0$. Without loss of generality we will assume that these subintervals are $I$ and $J$ again.  Our arguments will use the next two  claims

{\it Claim 1.} If there are $a,b,c\in\r$ such that $a+bf'(x)^2+ch'(x)^2=0$ for all $x\in I$, then either $abc\not=0$ or $a=b=c=0$. In the first case, we conclude that $h'(x)=\pm\sqrt{m_0+m_1f'(x)^2}$, where $m_0,m_1\not=0$, $m_0,m_1\in\r$.

The proof of the claim is as follows. According to the value of the constant $b$, we have two cases. If $b=0$, then $a+ch'^2=0$. In case that $c=0$, then $a=0$ and the claim is proved. If $c\not=0$, we deduce that $h'h''=0$, which is not possible. The other case is $b\not=0$. With a similar argument, we deduce $c\not=0$. If $a=0$, then $bf'(x)^2+ch'(x)^2=0$ for all $x\in I$, in particular, $bc<0$. Then $h'(x)=\pm\sqrt{-b/c}f'(x)$ so $h(x)=\pm\sqrt{-b/c}f(x)+m$, $m\in\r$. Thus $\alpha(x)=(x,f(x),\pm\sqrt{-b/c}f(x)+m)$ concluding that $\alpha$ is planar, which is contradictory. Thus $a\not=0$. Hence $h'^2=-a/c-b/cf'^2$ and the result follows by taking $m_0=-a/c$ and $m_1=-b/c$.
 
{\it Claim 2.} Suppose $h^{\prime
}=\pm \sqrt{m_{0}+m_{1}f^{\prime 2}}$, where $m_{0},m_{1}\not=0$.  Then the functions  $\left\{ 1,f^{\prime 2},f^{\prime }h^{\prime
}\right\} $ are linearly independent.

The proof is the following. Since $f^{\prime \prime }\neq 0$, let us introduce $s=f^{\prime }$.  Then the Wronskian of the set $\left\{
1,s^{2},\pm s\sqrt{m_{0}+m_{1}s^{2}}\right\} $ is $\mp \frac{2m_{0}^{2}}{%
\left( m_{0}+m_{1}s^{2}\right) ^{3/2}},$ and this proves the claim.

We come back to the proof of theorem. Dividing \eqref{eq22} by $Q=1+f^{\prime 2}+h^{\prime 2}$ and differentiating with respect to $x$, we obtain  a polynomial equation on $g'$
\begin{equation*}
\sum_{n=0}^{3}P_{n}(x) g^{\prime n}=0,
\end{equation*}%
where 
\begin{equation*}
\begin{split}
P_{0}(x)&=\left( \frac{h''-2(v_3-v_1h')(1+h'^2) }{Q  }\right)' \\
P_{1}(x)  &=\left( \frac{-f''-2 f' \left(3 v_1 h'^2-2 v_3 h'+v_1\right)+2 v_2 h'^2+2 v_2 }{Q}\right)'\\
P_{2}(x)  &=\left( \frac{h''-2 f'^2 \left(v_3-3 v_1 h'\right)-4 v_2 f' h'+2 v_1 h'-2 v_3}{Q}\right)' \\
P_{3}(x)  &=\left( \frac{-f''+2(v_2-v_1f')(1+f'^2)}{Q}\right)'.
\end{split}
\end{equation*}
Thus there are real constants $p_n\in\r$, $0\leq n\leq 3$, such that
\begin{equation}\label{pp}
\begin{split}
 h''-2(v_3-v_1h')(1+h'^2)&=p_0 Q \\
 -f''-2 f' \left(3 v_1 h'^2-2 v_3 h'+v_1\right)+2 v_2 h'^2+2 v_2 &=p_1 Q\\
 h''-2 f'^2 \left(v_3-3 v_1 h'\right)-4 v_2 f' h'+2 v_1 h'-2 v_3&=p_2 Q \\
 -f''+2(v_2-v_1f')(1+f'^2)&=p_3 Q.
\end{split}
\end{equation}
In order to simplify the notation, set 
$$c_1=p_0-p_2,\quad c_2=p_1-p_3.$$
We distinguish two cases.
\begin{enumerate}
\item Case $v_3=0$. There are two subcases.
\begin{enumerate}
\item Subcase $v_1=0$. After a dilation of $\r^3$, we can assume $\vec{v}=(0,1,0)$. Equations \eqref{pp} are
\begin{eqnarray}
h''&=&p_0 Q, \label{p0}\\
-f''+2(1+h'^2)&=&p_1 Q, \label{p1}\\
h''-4f'h'&=& p_2Q, \label{p2}\\
-f''+2(1+f'^2)&=&p_3 Q. \label{p3}
 \end{eqnarray}
 Combining \eqref{p1} and \eqref{p3} and using the value of $Q$, we have
 $$c_1+(c_1+2)f'^2+(c_1-2)h'^2=0.$$
 Because the three coefficients are distinct from $0$, Claim 1 implies that $h'= \pm\sqrt{m_0+m_1f'^2}$, where $m_0,m_1\not=0$. Using now \eqref{p0} and \eqref{p2}, $4f'h'=c_1Q$, or equivalently, 
 $$c_2+m_0c_2+(c_2+c_2m_1) f'^2 \mp 4f'h'=0.$$
From Claim 2, the functions $\{1,f'^2,f'h'\}$ are linearly independent, hence the coefficients must vanish,  obtaining a contradiction.
 
 \item Subcase $v_1\not=0$. Since  $\vec{v}=(v_1,v_2,0)$, after a dilation of $\r^3$, we can assume that $\vec{v}=(1,v_2,0)$. Now \eqref{pp} is
 \begin{eqnarray}
h''+2h'(1+h'^2)&=&p_0 Q,\label{p00}\\
-f''+2v_2(1+h'^2)-2f'(1+3h'^2)&=&p_1 Q,\label{p01}\\
h''+(2-4v_2 f'+6f'^2)h'&=& p_2Q,\label{p02}\\
-f''+2(v_2-f')(1+f'^2)&=&p_3 Q\label{p03}.
 \end{eqnarray}
 Combining  \eqref{p00} and \eqref{p02},
 \begin{equation}\label{mm}
 2h'(2v_2f'-3f'^2+h'^2)=c_1Q=c_1(1+f'^2+h'^2).
 \end{equation}
 Hence we can get the expression   
 $$h'^2=\frac{-c_1(1+f'^2)+4v_2f'h'-6h'f'^2}{c_1-2h'}.$$
 From \eqref{p01} and \eqref{p03}, we have 
$$
 -2v_2f'^2+2f'^3+(2v_2-6f')h'^2=c_2 Q=c_2(1+f'^2+h'^2).$$
 Substituting the above value of $h'^2$, 
 $$(-4 c_1 f'^3+2 c_1 v_2 f'^2-3 c_1 f'+c_1 v_2)+h'(2 c_2 v_2 f' -4 c_2 f'^2 -c_2 -4 v_2^2 f' +16 v_2 f'^2 -16 f'^3)=0.$$
 In this polynomial equation on $h'$ of degree $\leq 1$, if the coefficient of $h'$ is $0$, then $-4 c_1 f'^3+2 c_1 v_2 f'^2-3 c_1 f'+c_1 v_2=0$. This is  a polynomial on $f'$ and the leading coefficient of $f'^3$ is not $0$, which is not possible. Thus the coefficient of $h'$ is not $0$, obtaining
 $$h'=-\frac{2 c_1 v_2 f'^2-4 c_1 f'^3-3 c_1 f'+c_1 v_2}{2 c_2 v_2 f' -4 c_2 f'^2 -c_2 -4 v_2^2 f' +16 v_2 f'^2 -16 f'^3}.$$
 Substituting into \eqref{mm}, and after some manipulations, we have an expression of type 
 $$\sum_{n=0}^9 A_n f'(x)^n=0,$$
 where $A_n$ are real constants.  Thus all coefficients $A_n$ must vanish. However, the computation of  $A_9$ gives $A_9=-512$. This completes the proof for the case $v_3=0$.  
 
\end{enumerate}

\item Case $v_3\not=0$. After a dilation, we suppose $\vec{v}=(v_1,v_2,1)$. We distinguish four subcases.  
\begin{enumerate}
\item  Subcase $v_{1}=v_2=0$.   Then \eqref{pp} is%
\begin{eqnarray}
h^{\prime \prime }-2\left( 1+h^{\prime 2}\right) &=&p_{0}Q  \label{p10} \\
-f^{\prime \prime }+4f^{\prime }h^{\prime } &=&p_{1}Q  \label{p11} \\
h^{\prime \prime }-2f^{\prime 2}-2 &=&p_{2}Q \label{p12} \\
-f^{\prime \prime } &=&p_{3}Q. \label{p13}
\end{eqnarray}
We   deduce from \eqref{p10} and \eqref{p12} that
$c_1Q=2f^{\prime 2}-2h^{\prime 2}$ or equivalently, by  $Q=1+f'^2+h'^2$, 
\begin{equation*}
c_{1}+\left( c_{1}-2\right) f^{\prime 2}+\left( c_{1}+2\right) h^{\prime
2}=0.
\end{equation*}%
Clearly  the coefficients are
nonzero and Claim 1 implies $h^{\prime }=\pm \sqrt{m_{0}+m_{1}f^{\prime 2}%
}$, $m_0, m_1\not=0$. Moreover from \eqref{p11} and \eqref{p13}, we deduce $c_2
Q=4f^{\prime }h^{\prime },$  and substituting $h'^2=m_0+m_1 f'^2$,  
\begin{equation*}
c_2 \left( 1+m_{0}\right) +c_2\left( 1+m_{1}\right) f^{\prime 2}\mp 4f^{\prime }h'=0,
\end{equation*}%
which gives a contradiction from Claim 2.
\item  Subcase  $v_{1}=0$ and $v_{2}\neq 0$.  Then \eqref{pp} is%
\begin{eqnarray}
h^{\prime \prime }-2\left( 1+h^{\prime 2}\right) &=&p_{0}Q  \label{p20} \\
-f^{\prime \prime }+4f^{\prime }h^{\prime } +2v_2(1+h'^2)&=&p_{1}Q  \label{p21} \\
h^{\prime \prime }-2f^{\prime 2}-4v_{2}f^{\prime }h^{\prime }-2 &=&p_{2}Q 
\label{p22} \\
-f^{\prime \prime }+2v_{2}\left( 1+f^{\prime 2}\right) &=&p_{3}Q. 
\label{p23}
\end{eqnarray}%
It follows from \eqref{p20} and \eqref{p22} that 
\begin{equation} \label{p24} 
c_1 Q=2f^{\prime 2}-2h^{\prime 2}+4v_{2}f^{\prime
}h^{\prime }.
\end{equation}
Similarly from \eqref{p21} and \eqref{p23}, 
\begin{equation}\label{p25} 
c_2 Q=4f^{\prime }h^{\prime }+2v_{2}(h'^2-f'^2).
\end{equation}%
Combining \eqref{p24} and \eqref{p25}, we deduce 
\begin{equation*}
2  ( 1+v_{2}^{2}) (f^{\prime 2}-h^{\prime
2})=\left(c_1 -c_2 v_{2}  \right)
Q=\left(c_1 -c_2 v_{2}  \right)(1+f'^2+h'^2). 
\end{equation*}%
Let $c_{3}=p_{0}-p_{2}-v_{2}\left( p_{1}-p_{3}\right) $. This equation writes as 
\begin{equation*}
c_{3}+(c_3-2(1+v_2^2))f'^2 +(c_3+2(1+v_2^2))h'^2=0,
\end{equation*}%
where the coefficients of $\{1,f'^2,h'^2\}$ are clearly nonzero. Then Claim 1 implies $h^{\prime
}=\pm \sqrt{m_{0}+m_{1}f^{\prime 2}}$, $m_0, m_1\not=0$.  Coming back to \eqref{p24}, we derive%
\begin{equation*}
4v_2f^{\prime }h^{\prime }=c_1 \left( 1+m_{0}\right)+2m_0+
 \left(c_1 \left( 1+m_{1}\right)
+2m_1-2\right) f^{\prime 2}.
\end{equation*}%
Claim 2 concludes that this subcase is false because $v_2\not=0$

\item  Subcase  $v_{1}\neq 0$ and $v_{2}=0$.  Then \eqref{pp} is%
\begin{eqnarray}
h^{\prime \prime }-2\left( 1-v_{1}h^{\prime }\right) \left( 1+h^{\prime
2}\right) &=&p_{0}Q \label{p30}\\
-f^{\prime \prime }-2f^{\prime }\left( 3v_{1}h^{\prime 2}-2h^{\prime
}+v_{1}\right) &=&p_{1}Q  \label{p31} \\
h^{\prime \prime }-2f^{\prime 2}\left( 1-3v_{1}h^{\prime }\right) +2\left(
v_{1}h^{\prime }-1\right) &=&p_{2}Q  \label{p32} \\
-f^{\prime \prime }-2v_{1}f^{\prime }\left( 1+f^{\prime 2}\right) &=&p_{3}Q.
\label{p33}
\end{eqnarray}%
From \eqref{p30} and \eqref{p32} we derive $f'^2=A/B$,  where%
\begin{equation*}
A=\left( 2v_{1}h^{\prime }-c_{1}-2\right) h^{\prime 2}-c_{1},\quad B=c_{1}-2\left( 1-3v_{1}h^{\prime }\right) .
\end{equation*}
Let us observe that $B\not=0$ because $h''\not=0$ and $v_1\not=0$.
Similarly, from \eqref{p31} and \eqref{p33}, we have 
\begin{equation*}
2f'(v_1f'^2+2h'^2-3v_1h'^2)=c_2Q=c_{2}\left( 1+f^{\prime 2}+h^{\prime 2}\right).
\end{equation*}%
Substituting $f^{\prime }$ by $\pm \sqrt{A/B}$,  
\begin{equation*}
c_{2}\sqrt{B}\left( A+B\left( 1+h^{\prime 2}\right) \right) =\pm 2\sqrt{A}%
(Av_{1}+Bh'(2-3v_1h')).
\end{equation*}%
After squaring both sides,  we obtain%
\begin{equation*}
\sum_{n=0}^{9}A_{n}h'(x)^n=0,
\end{equation*}%
where $A_n$ are real constants. Being $A_{9}=2^{11}v_{1}^{5}\neq 0$, we arrive to a contradiction.

\item  Subcase  $v_{1}v_{2}\neq 0$. Then \eqref{pp}  writes%
\begin{eqnarray}
h^{\prime \prime }-2\left( 1-v_{1}h^{\prime }\right) \left( 1+h^{\prime
2}\right)  &=&p_{0}Q  \label{p40} \\
-f^{\prime \prime }-2f^{\prime }\left( 3v_{1}h^{\prime 2}-2h^{\prime
}+v_{1}\right) +2v_{2}\left( 1+h^{\prime 2}\right)  &=&p_{1}Q \label{p41} \\
h^{\prime \prime }-2f^{\prime 2}\left( 1-3v_{1}h^{\prime }\right)
-4v_{2}f^{\prime }h^{\prime }+2\left( v_{1}h^{\prime }-1\right)  &=&p_{2}Q 
\label{p42} \\
-f^{\prime \prime }+2\left( v_{2}-v_{1}f^{\prime }\right) \left( 1+f^{\prime
2}\right)  &=&p_{3}Q.  \label{p43}
\end{eqnarray}%
 From \eqref{p40} and \eqref{p42},
\begin{equation*}
2\left( v_{1}h^{\prime
}-1\right) h^{\prime 2}+2f^{\prime 2}\left( 1-3v_{1}h^{\prime }\right)
+4v_{2}f^{\prime }h^{\prime }=c_1 Q=c_{1}\left( 1+f^{\prime 2}+h^{\prime 2}\right),
\end{equation*}%
or equivalently%
\begin{equation*}
Af^{\prime 2}+Bf^{\prime }+C=0,
\end{equation*}%
where 
\begin{equation*}
\begin{split}
A &=c_{1}-2+6v_{1}h^{\prime },\\
B&=-4v_{2}h^{\prime }, \\
C &=c_{1}+\left( c_{1}+2-2v_{1}h^{\prime }\right) h^{\prime 2}.
\end{split}
\end{equation*}
Note that $A$ and $B$ cannot vanish because $v_{1}v_{2}\not=0$ and $h^{\prime \prime
}\neq 0$.  Then  
$$f^{\prime }=\frac{-B\pm \sqrt{B^{2}-4AC}}{2A}.$$
 On
the other hand,  \eqref{p41} and \eqref{p43} imply%
\begin{equation}\label{tt}
c_{2} Q +2 \left(f'(3 v_1  h'^2-2  h')+f'^2(v_2-v_1 f')-v_2 h'^2\right)=0.
\end{equation}
As in previous subcase, our purpose is to substitute the value of $f'$ in order to obtain a polynomial on $h'$.   Let $D=B^2-4AC$. If we write \eqref{tt} again in terms of $\sqrt{D}$, then we have a polynomial equation of type $a+b\sqrt{D}=0$, hence, $a^2-b^2D=0$. Substituting the value of $D$ as well as of $A$, $B$ and $C$, we obtain  the desired polynomial equation on $h'$, namely, 
\begin{equation*}
\sum_{n=0}^{12}A_{n}h'^n=0,
\end{equation*}%
where $A_n$ are real constants. Because $A_{12}=2^{16}3^{3}v_{1}^{8}\neq 0$,  we arrive to a contradiction,
completing the proof of theorem.
\end{enumerate}

\end{enumerate}

\end{proof}

\section{Translating solitons of homothetical type}

Let $u(x,y)=f(x)g(y)$ be a homothetical surface. Suppose that the surface is also a translating soliton with respect to $\vec{v}$. There are three initial  cases that can be previously considered.
 \begin{enumerate}
 \item Case that  $f$ or $g$ is constant. Then the surface is ruled, so we know that the surface is a grim reaper or a plane parallel to the vector $\vec{v}$. 
 \item Case that $f$ (or $g$) is linear. Indeed, if $f(x)=ax+b$ with $a,b\in\r$, $a\not=0$, then \eqref{eqh1} is a polynomial equation $\sum_{n=0}^3 A_n(y)x^n=0$. In particular all coefficients $A_n$ must vanish. The computation of $A_3$ yields  $2a^3v_2 g'^3$. Since $g$ is not constant and $a\not=0$, we deduce $v_2=0$. Now the computation of $A_2$ gives $A_2=2a^3(av_1g)g'^2$. Hence, $v_1=v_3=0$, obtaining a contradiction.
\item Case that $f$ and $f'$ (or $g$ and $g'$) are linearly dependent.  Assume that $f^{\prime }=af,$ $a\in\r$, $a\not=0$.  Then $f''=a^2f$ and \eqref{eqh1} is a  polynomial equation on $f$ of
degree $3$, namely,  $\sum_{n=0}^3 A_n(y) f^n=0$. Then all coefficients $A_n$ must vanish. In particular, $A_0=-2v_3$, hence $v_3=0$. Now $A_1=0$ and $A_3$ lead to
\begin{eqnarray*}
A_1=a\left( a+2v_{1}\right) g+2v_{2}g^{\prime }+g^{\prime \prime } &=&0  \\ 
A_3=\left( av_{1}g+v_{2}g^{\prime }\right) \left( ag^{2}+g^{\prime 2}\right)
+a^{2}\left( g^{2}g^{\prime \prime }-gg^{\prime 2}\right)  &=&0.  
\end{eqnarray*}%
The linear combination $A_1-a^2g A_3=0$  writes $-a^4g^3+(2av_1-a^2)gg'^2+2v_2g'^3=0$, or equivalently, 
$$-a^3+(2av_1-a^2)\left(\frac{g'}{g}\right)^2+2v_2\left(\frac{g'}{g}\right)^3=0.$$
Since all coefficients are not all zero, there is $b\in\r$ such that $g'/g=b$ with $b\not=0$. Taking into account that $g''=b^2g$ and that $a^2+b^2\not=0$, then $A_1=0$ is  
$$a^2+ (2(av_1+bv_2)+b^2)g=0,$$
obtaining a contradiction.
\end{enumerate}
After this discussion,  we can assume that $f'f''\not=0$ and $g'g''\not=0$ in their domains and let us introduce new variables. So, let  $p=p(f)=f'$, as well as,  $q=q(g)=g'$. Then $p'=f''/f'$ and $q''=g''/g'$. Let us observe that $pp'qq'\not=0$. Now the translating soliton equation \eqref{eqh1} is 
\begin{equation}\label{eqh2}
(1+f^2q^2)gpp'-2fgp^2q^2+(1+g^2p^2)fqq'-2(1+p^2g^2+f^2q^2)(-v_1pg-v_2fq+v_3)=0.
\end{equation}
We now give a partial result on Problem 2 in case that $\vec{v}$ is one of the canonical basis of $\r^3$.

\begin{theorem} \label{t3} Grim reapers are the only translating solitons of homothetical type when $\vec{v}$ is one vector of the canonical basis.
\end{theorem}

\begin{proof} We know that the case $\vec{v}=(0,0,1)$ was solved in \cite{Lopez3}. It remains the case that $\vec{v}$ is $(1,0,0)$ or $(0,1,0)$. By the symmetry of Equation \eqref{eqh2} with respect to $v_1$ and $v_2$, it suffices to consider the case that $\vec{v}=(0,1,0)$. Then 
\begin{equation}\label{eqh22}
(1+f^2q^2)gpp'-2fgp^2q^2+(1+g^2p^2)fqq'+2fq(1+p^2g^2+f^2q^2)=0.
\end{equation}
We divide by $fgp^2q^2$,
$$
\frac{1}{q^{2}}\left( \frac{p^{\prime }}{fp}\right) +\frac{1}{p^{2}}\left(\frac{q^{\prime
}}{gq}\right) +\frac{1}{q}\left( gq^{\prime }-q\right) +\frac{1}{p}\left(
fp^{\prime }-p\right)   +2\left(
\frac{1}{p^{2}gq}+\frac{g}{q}+\frac{f^{2}q}{p^{2}g}\right)  =0.  
$$
Differentiating with respect to $f$ and $g$ successively, we obtain%
$$
\left( \frac{1}{q^{2}}\right) ^{\prime }\left( \frac{p^{\prime }}{fp}\right) ^{\prime
}+\left( \frac{1}{p^{2}}\right) ^{\prime }\left( \frac{q^{\prime }}{gq}\right) ^{\prime } +2\left[ \left( \frac{1}{p^{2}}\right) ^{\prime }\left( \frac{1}{gq}\right) ^{\prime
}+\left( \frac{f^{2}}{p^{2}}\right) ^{\prime }\left( \frac{q}{g}\right) ^{\prime }\right]    =0. 
$$
Notice that $\left( \frac{1}{p^{2}}\right) ^{\prime }\left( \frac{1}{q^{2}}\right) ^{\prime
}\neq 0$. Dividing by $2\left( \frac{1}{p^{2}}\right) ^{\prime }\left(
\frac{1}{q^{2}}\right) ^{\prime }$ and next differentiating with respect to $f$ and $g$
successively, 
\begin{equation}\label{eqh3}\left( \dfrac{%
\left( \dfrac{f^{2}}{p^{2}}\right) ^{\prime }}{\left( \dfrac{1}{p^2}\right) ^{\prime }}%
\right) ^{\prime }\left( \dfrac{\left( \dfrac{q}{g}\right) ^{\prime }}{%
\left( \dfrac{1}{q^2}\right) ^{\prime }}\right) ^{\prime }   =0.   
\end{equation}%

\begin{enumerate}
\item Case   
$$\left(\frac{f^2}{p^2}\right)'=a\left(\frac{1}{p^2}\right)'$$
for some $a\not=0$. Integrating we have   $p^2=k f^2-a k$ for some $k\not=0$. Differentiating with respect to $f$, we deduce $pp'=k f$. Substituting into \eqref{eqh22}, we have a polynomial equation $B_1(g)f+B_3(g)f^3=0$. Thus $B_1=B_3=0$. The computation of these coefficients yield
\begin{eqnarray*}
B_1&=&-q \left(a g^2 k-1\right) \left(q'+2\right)+2 a g k q^2+g k,\\
B_3&=&q \left(g^2 k \left(q'+2\right)-g k q+2 q^2\right).
\end{eqnarray*}
Equation $B_3=0$ can be solved explicitly. Suppose $k>0$ (an analogous  argument  if $k<0$). Then
$$q(g)=g \sqrt{k} \tan \left(m-\frac{2}{\sqrt{k}} \log (g)\right),\quad m\in\r.$$
Substituting into $B_1=0$, we conclude
$$2 \left(a g^2 k-1\right) \tan ^3\left(m-\frac{2 \log (g)}{\sqrt{k}}\right)+\sqrt{k} \left(a g^2 k+1\right) \tan ^2\left(m-\frac{2 \log (g)}{\sqrt{k}}\right)+\sqrt{k}=0,$$
obtaining a contradiction.

\item   Case
$$\left( \dfrac{q}{g}\right) ^{\prime }=a\left( \dfrac{1}{q^{2}}\right)
^{\prime }$$
for some $a\not=0$. Integrating, 
$$g=\frac{q^{3}}{a+kq^{2}}, $$
for some constant $k\in\r$.  Differentiating with respect to $g$,
$$q^{\prime }=\frac{q^{2}\left( 3a+kq^{2}\right) }{\left( a+kq^{2}\right) ^{2}}.$$
By substituting these values of $g$ and $q$ in \eqref{eqh22},  we obtain  
\begin{equation*}
\frac{q}{(a+kq^2)^4}\sum_{n=0}^{10}C_{n}(x)q^{n}=0.
\end{equation*}%
Then all coefficients $C_n$ must vanish. However, the computation of $C_0$ gives $C_{0}=2a^{4}f$ which is not possible.  
\end{enumerate}
\end{proof}

\section{Extension of the results to the Lorentzian setting}\label{s4}

In this last section, we extend the results to the Lorentz-Minkowski 3-space 
$\mathbb{R}_{1}^{3}$. Here $\r_1^3$ is the affine space $\r^3$ endowed with the canonical Lorentzian metric $%
dx^{2}+dy^{2}-dz^{2}$.  Denote $\left\langle \cdot ,\cdot \right\rangle _{L}$
and $\times _{L}$ the Lorentzian inner and cross product, respectively.

We first consider a (non-planar) non-degenerate cylindrical surface $X(s,t)=\alpha (s)+t%
w$ where $\alpha =\alpha (s) $ is parametrized by the arc length $s$ and $w\in \mathbb{R}_{1}^{3}$.  The unit normal vector $N$ is parallel to $\alpha
^{\prime }(s)\times _{L}w$ and hence \eqref{eq1} writes  
\begin{equation}\label{38}
\left\langle w,w\right\rangle _{L}\left\langle \alpha'\times_L w,\alpha''\right\rangle _{L}=2\epsilon
\left( \epsilon _{1}\left\langle w,w\right\rangle _{L}-\langle
\alpha ^{\prime }(s),w\rangle _{L}^{2}\right) \left\langle \alpha
^{\prime }(s)\times_L w,\vec{v}\right\rangle _{L},   
\end{equation}%
 where $\epsilon $ is the sign of $\left\langle \alpha ^{\prime }(s)\times_L 
w,\alpha ^{\prime }(s)\times_L w\right\rangle _{L}$ and $\epsilon
_{1}=\left\langle \alpha ^{\prime }(s),\alpha ^{\prime }(s)\right\rangle
_{L}$.  In case that the rulings are lightlike, the surface is  a translating soliton if    $\langle \alpha
^{\prime }(s)\times_L w,\vec{v}\rangle _{L}=0$.  In particular, this equation holds if $\vec{v}$ is parallel to $w$ being $\alpha$ is an arbitrary curve. 

In all Lorentzian versions of the results, we will conclude that the surface is a cylindrical surface. According to the causal character of the rulings, the description of the translating solitons of $\r_1^3$ of cylindrical type is the following (\cite{al}).  After a rigid motion of $\r_1^3$, we can fix $w$.
\begin{enumerate}
\item Spacelike rulings. Let $w=(1,0,0)$. If $X(s,t)=(0,s,u(s))+tw$, then \eqref{eq1}
$$u''=\left\{\begin{array}{ll}
2(1-u'^2)(v_2 u'-v_3), & 1-u'^2>0\\
-2(1-u'^2)(v_2 u'-v_3), & 1-u'^2<0.
\end{array}\right.$$
For example, if $\vec{v}=(0,0,1)$, the rulings are orthogonal to $\vec{v}$ and the integration of both equations give 
$$u(s)=\left\{\begin{array}{ll}
-\frac12\log(\cosh(-2s+a))+b,& 1-u'^2>0\\
\frac12\log(\sinh(2s+a))+b,& 1-u'^2<0,
\end{array}\right.$$
where $a,b\in\r$. These two curves appeared in \cite{castro}. 
\item Timelike rulings. Let $w=(0,0,1)$. If $X(s,t)=(s,u(s),0)+tw$, then the surface is timelike and Equation \eqref{eq1} is 
$$u''=2(1+u'^2)(v_1 u'-v_2).$$
If   $\vec{v}=(0,1,0)$, the rulings are orthogonal to $\vec{v}$ and the solution is $u(s)=\log(\cos (2s+a))/2+b$, $a,b\in\r$. 
 \item Lightlike rulings. Then $H=0$, so the translating equation \eqref{eq1} is $\langle N,\vec{v}\rangle=0$.  Let $w=(1,0,1)$ and 
 $X(s,t)=(u(s),s,-u(s))+tw$. The surface is not degenerated if $u'\not=0$. Then   Equation \eqref{eq1} is 
 $$v_1-2v_2u'-v_3=0.$$
 If $v_2=0$, then $\vec{v}$ is parallel to $w$ and with arbitrary generating curve.  Otherwise the function $u$ is linear and $X(s,t)$ is a plane.  
 \end{enumerate}

Summarizing, the cylindrical translating solitons in $\r_1^3$  are planes (when the rulings are parallel to $\vec{v}$), Lorentzian grim reapers and cylindrical surfaces whose rulings are lightlike and parallel to $\vec{v}$.

As we have pointed out,   the extensions of Theorems \ref{t1}, \ref{t2} and \ref{t3}  to the Lorentzian setting are straightforward and the conclusion is that the surfaces must be cylindrical surfaces.  

\begin{theorem}\label{t5} A translating soliton in $\mathbb{R}_{1}^{3}$   that is the sum of two curves and where one of the generating curve is planar  must be a cylindrical surface.  
\end{theorem}

\begin{theorem}
\label{t6} A   translating
soliton in $\mathbb{R}_{1}^{3}$ of homothetical type when $\vec{v}$ is one
vector of the canonical basis must be a cylindrical surface.
\end{theorem}


\section*{Declarations}

\noindent{\bf Funding.} Rafael L\'opez has been partially supported by the grant no. MTM2017-89677-P, MINECO/ AEI/FEDER, UE.

\noindent{\bf Code availability.}  Not applicable.

\noindent{\bf Contributions.} The authors equally conceived of the study, participated in its design and coordination, drafted the manuscript, participated in the sequence alignment, and read and approved the final manuscript.

\noindent{\bf  Conflict of interest.} The authors declare that there is no conflict of interest.

\end{document}